\documentclass[12 pt]{amsart}

\usepackage{amscd,amsfonts,amssymb,amsmath}
\usepackage{tikz-cd}
\usepackage{xcolor}
\usepackage{placeins}
\usepackage{hyperref}

\newtheorem{theorem}{Theorem}
\theoremstyle{definition}

\newtheorem{corollary}[theorem]{Corollary}
\newtheorem{lemma}[theorem]{Lemma}
\newtheorem{proposition}[theorem]{Proposition}
\newtheorem{question}[theorem]{Question}

\newtheorem{remark}[theorem]{Remark}
\numberwithin{equation}{subsection}

\title{Links and dynamics}
\author[Valeriy Bardakov, Tatyana Kozlovskaya, and  Olga Pochinka]{Valeriy Bardakov, Tatyana Kozlovskaya, and Olga Pochinka}
\date{\today}

\begin{document}
\maketitle
\begin{abstract}
Knots naturally appear in continuous dynamical systems as flow periodic trajectories. However, discrete dynamical systems are also closely connected  with the theory of knots and links. For example, for Pixton diffeomorphisms, the equivalence class of the Hopf knot, which is the orbit space of the unstable saddle separatrix in the manifold $\mathbb{S}^2\times \mathbb{S}^1$, is a complete invariant of the topological conjugacy of the system. In this paper we distinguish a class of three-dimensional Morse-Smale diffeomorphisms for which the complete invariant of topological conjugacy is the equivalence class of a link in $\mathbb{S}^2\times \mathbb{S}^1$.

We proved that if  $M$ is a  link complement in $\mathbb{S}^3$  (in particular, is $\mathbb{S}^3$), or a 
handlebody $H_g$ of genus $g \geq 0$, or
closed, connected, orientable 3-manifold, then the set of equivalence classes of tame links in $M$ is countable.
As corollary we get  that in $\mathbb{S}^2\times \mathbb{S}^1$ there exists a countable number of equivalence classes of tame links. It is proved that any essential  link  can be realized by a diffeomorphism of the class under consideration.


\end{abstract}

\maketitle
\tableofcontents

\section{Introduction}
Dynamicists use knot and link invariants to describe periodic orbits of flows, that helps them better understand the underlying ODEs. However, specialists in discrete dynamical systems also actively use the theory of knots and links to describe the invariants of dynamics. In particular, the dynamics of some Morse-Smale diffeomorphisms, given on an arbitrary 3-manifold, can be completely determined (up to topological conjugacy) by the behavior of separatries of saddle points, which in turn is described by a knot or a link in the manifold $\mathbb{S}^2\times \mathbb{S}^1$. 

Notice, that the  topological classification of arbitrary Morse-Smale 3-diffeomorphisms was done in \cite{BGP} and it uses as a complete invariant the equivalence class of a pair of two-dimensional laminations in some unknown 3-manifold. Since the classification (up to a homeomorphism) of 3-manifolds is an open problem, it is a natural wish to select diffeomorphisms that admit well-studied objects as invariants. For example, for a class of 3-diffeomorphisms whose non-wandering set consists of four hyperbolic fixed points, the complete invariant of topological conjugacy is  equivalence class of a Hopf knot in $\mathbb{S}^2\times \mathbb{S}^1$ \cite{BoGr}, \cite{PoTaSh}. In the present paper we  distinguish a class of Morse-Smale 3-diffeomorphisms for which the complete invariant of topological conjugacy is the equivalence class of an essential link in $\mathbb{S}^2\times \mathbb{S}^1$.  

In more details. Let $f \colon M^3\to M^3$ be a Morse-Smale diffeomorphism, given on a closed orientable connected 3-manifold. Let $\Omega_q,\,q\in\{0,1,2,3\}$ be the set of all its periodic points $p$ with $\dim\, W^u_p=q$.
Using $|\Omega_q|$ we denote the number of points in the set $\Omega_q$. Note that $|\Omega_0|\geqslant 1,\,|\Omega_3|\geqslant 1$ (see, for example, \cite{S3}). At the same time, if $|\Omega_0|= 1$ or $|\Omega_3|=1$, then $M^3\cong\mathbb S^3$ \cite{GGZP}. Everywhere below we assume that the set $\Omega_3$ consists of a single source $\alpha$. In this case, the set $\Omega_2$ is empty \cite{global}. If $\Omega_1=\emptyset$, then $f$ is a source-sink diffeomorphism and all such diffeomorphisms are pairwise topologically conjugate \cite{GrMePo2016}. If $|\Omega_1|=1$, then $f$ is a Pixton diffeomorphism and the topological classification of such diffeomorphisms is completely determined by the equivalence class of the Hopf knot  \cite{BoGr}. Everywhere below we assume that $|\Omega_1|>1$.  Renumber the saddle orbits of the set $\Omega_1$: $$\mathcal O_{\sigma_1},\dots,\mathcal O_{\sigma_k}$$ and denote by $m_i$ the period of $\sigma_i$. Then for $i\in\{1,\dots,k\}$, the set $S_i=W^s_{\sigma_i}\cup\alpha$ is homeomorphic to a 2-sphere topologically embedded into $\mathbb S^3$ \cite{GrMePo2016}. 

Denote by $G$ the set of Morse-Smale diffeomorphisms $f\colon\mathbb S^3\to\mathbb S^3$ such that 2-spheres $S_1,\dots,S_k$ bound 3-balls $B_1,\dots,B_k$ with pairwise disjoint interiors in $\mathbb S^3$. For $f\in G$ let $B=\bigcup\limits_{i=1}^k\left(\bigcup\limits_{j=0}^{m_i-1}f^{j}(B_i)\right)$. Then the set $\mathbb S^3\setminus B$ contains exactly one periodic point of the diffeomorphism $f$, which is a fixed sink \cite{GGZP}, let's denote it $\omega$ (see Fig. \ref{pi1}).
\begin{figure}[h!]\center{\includegraphics
[width=0.65\linewidth]{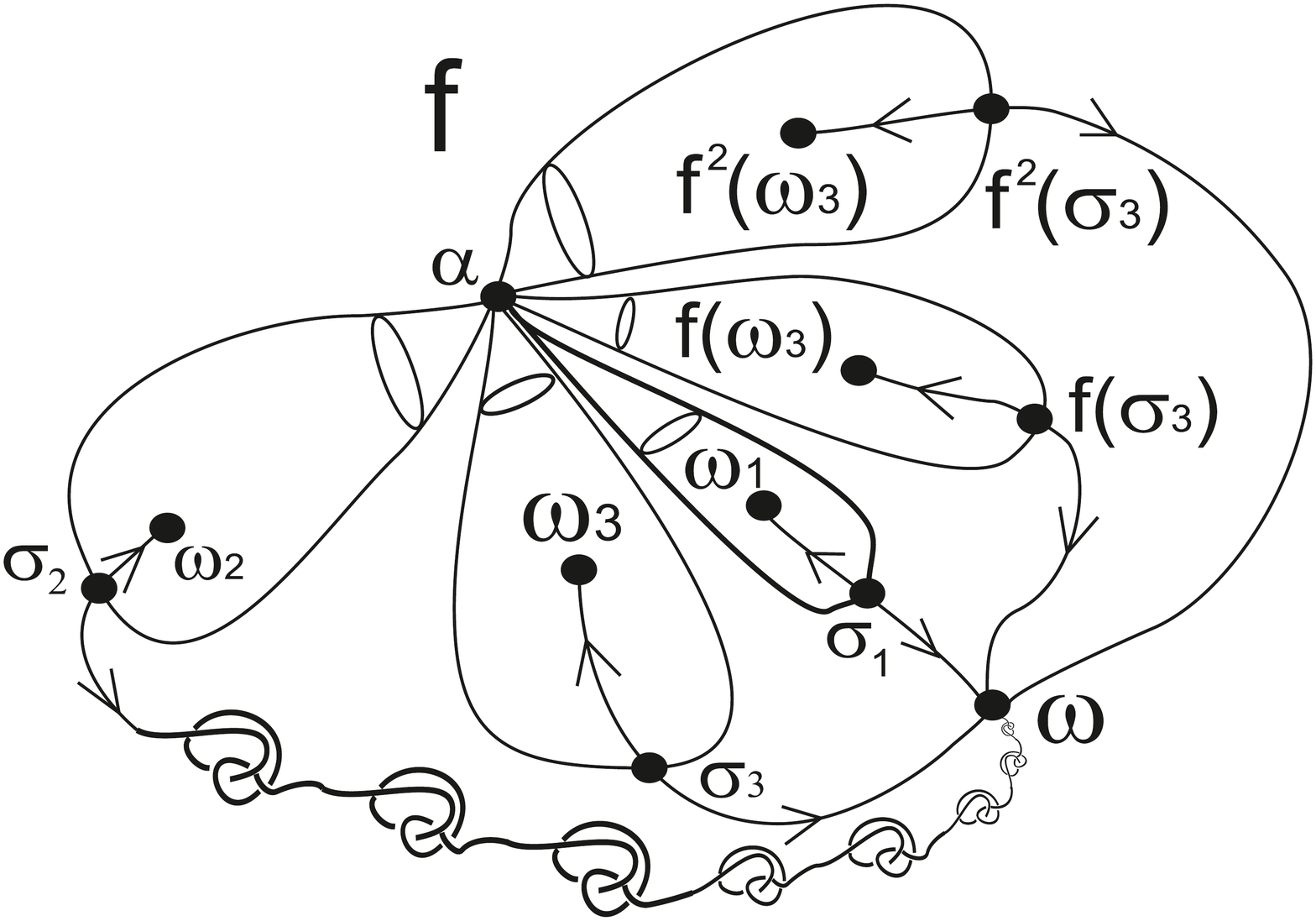}}\caption{Phase portrait of a diffeomorphism $f\in G$}
\label{pi1}
\end{figure}

Let ${\bf x}=(x_1,x_2,x_3)\in\mathbb R^3$, $|| {\bf x} ||=\sqrt{x_1 ^2 +x_2^2+x_3^2}$ and $h: \mathbb R^3\to\mathbb R^3$ be a diffeomorphism given by the formula $h({\bf x})=\frac{\bf x}{2}.$ Define a map $p\colon\mathbb R^3\setminus O\to\mathbb S^{2}\times\mathbb S^1$ by formula  
$$p({\bf x})=\left(\frac{x_1}{||{\bf x}||},\frac{x_2}{||{\bf x}||}, \log_2(||{\bf x}||)\pmod 1\right).
$$
Let $V_{\omega}=W^s_{\omega}\setminus\omega$.
Due to the hyperbolicity of the sink $\omega$ there is a diffeomorphism $\psi_{\omega}\colon V_{\omega}\to\mathbb R^3\setminus O$, which conjugates the diffeomorphisms $f$ and $h$. Let $p_{\omega}=p\psi_{_f}: V_{f}\to\mathbb S^{2}\times\mathbb S^1$ and $\mathcal K_f=p_{\omega}(W^u_{\Omega_1}\cap V_{\omega})$. According to \cite{GrMePo2016}, the set $\mathcal K_f$ is a smooth link, consisting of {\it essential} (non-contractible) knots in $\mathbb S^2\times\mathbb S^1$ (see Fig. \ref{pi2}).
\begin{figure}[h!]\center{\includegraphics
[width=0.45\linewidth]{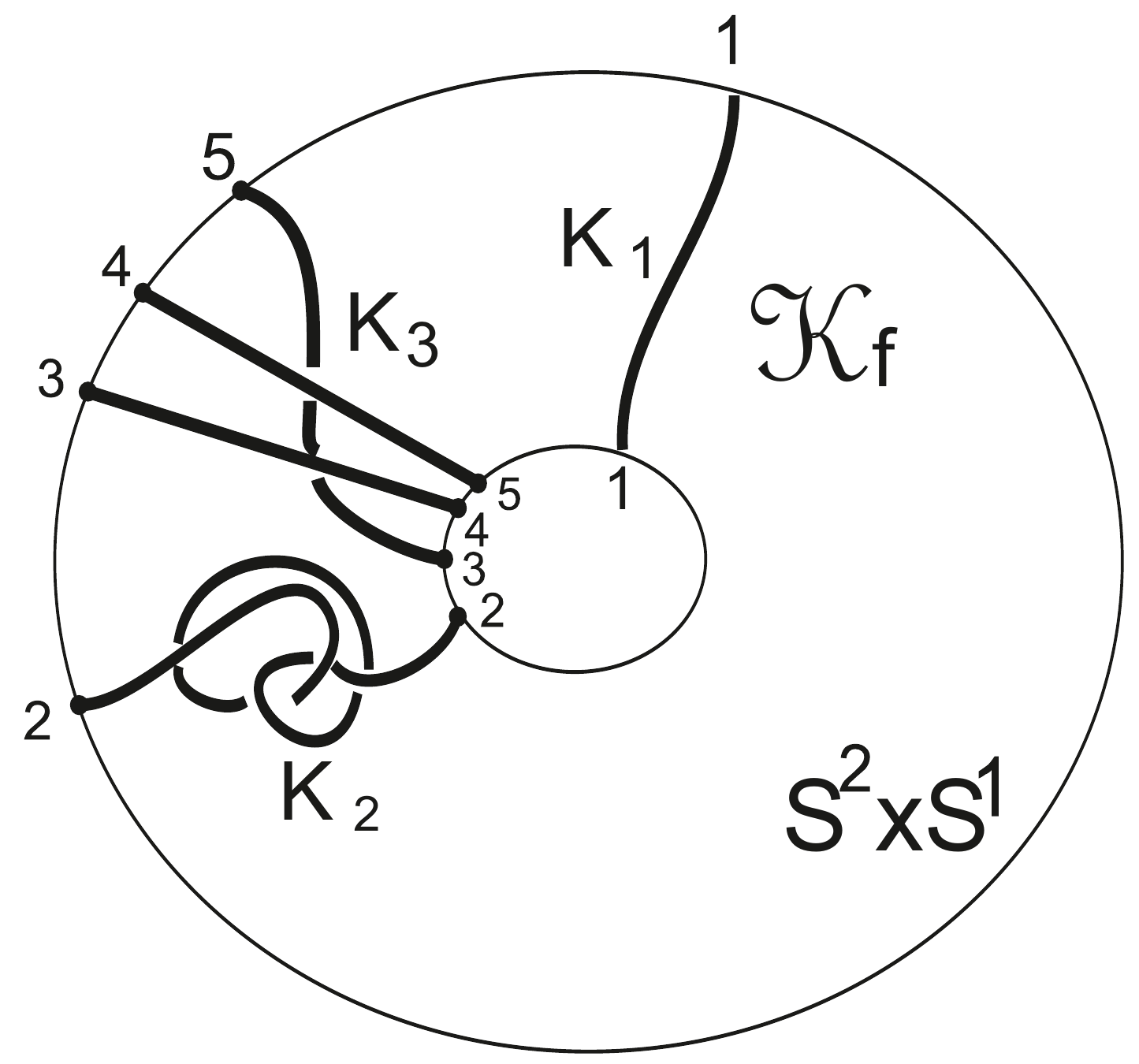}}\caption{Link $\mathcal K_f$ for the diffeomorphism $f\in G$ from figure \ref{pi1}}\label{pi2}
\end{figure}

Two links $\mathcal K_f,\,\mathcal K_{f'}$ are called {\it equivalent} if there is an orientation-preserving homeomorphism   $\varphi:\mathbb S^2\times\mathbb S^1\to\mathbb S^2\times\mathbb S^1$ such that $\varphi(\mathcal K_f)=\mathcal K_{f'}$.
\begin{theorem}\label{coni}
Two diffeomorphisms $f,\,f'\in G$ are topologically conjugated if and only if their links $\mathcal K_f,\,\mathcal K_{f'}$ are equivalent.
\end{theorem}
\begin{theorem}\label{rea}
For every smooth essential link $\mathcal K\subset\mathbb S^2\times\mathbb S^1$ there is a diffeomorphism $f\in G$ such that $\mathcal K_f$ is equivalent to $\mathcal K$.
\end{theorem}

Thus, we reduce the verification of the qualitative sameness of the continuum of the trajectories of two diffeomorphisms $f,\,f'\in G$ to the establishment of the coincidence of the equivalence classes of the corresponding links. A further natural step of this reduction is the description of the equivalence classes of links, and the first question in this direction is the number of these classes. 

In the present article we are considering only  {\it tame} links, that is composed by pairwise disjoint polygonal knots in an orientable 3-manifold $M$. Two links  $L$ and $L'$ in $M$ is said to be {\it equivalent} if there is an orientation-preserving homeomorphism $h \colon M \to M$ such that $h(L) = L'$. 

In this article  we prove the following result.
\begin{theorem}\label{mt}
Let $M$ be a 3-manifold from the following list

1) a link complement to $\mathbb{S}^3$, in particular, $\mathbb{S}^3$;

2) the handlebody $H_g$ of genus $g \geq 0$;

3) a closed connected, orientable 3-manifold.

Then the  set of  equivalence classes of tame links in $M$ is countable.
\end{theorem}

\begin{corollary}\label{cor}
By Theorem \ref{mt}, the set of  equivalence classes  of tame links in $\mathbb{S}^2 \times \mathbb{S}^1$ is countable. Using a series of pairwise non-equivalent Hopf knots  from \cite{AMP}, we get that there are infinitely many of equivalence classes of essential links in $\mathbb{S}^2 \times \mathbb{S}^1$.
\end{corollary}

\medskip

The paper is organised as follows. In Section \ref{prelim1} we recall need facts from dynamical systems,  from low dimensional topology, in particular, from knot theory and braid theory.

In Section \ref{Dyn} we prove Theorem \ref{coni} and \ref{rea}, using complete classification of Morse-Smale 3-diffeomorphisms given in \cite{BGP}.

In Section \ref{Sph} we study links in 3-sphere. Using the fact that any such link is the closure of a braid and the set of braids is countable, we prove Theorem \ref{mt} for the case $M^3 = \mathbb{S}^3$. As corollary we get that the set of equivalence classes of finite spatial graphs in $\mathbb{S}^3$ is countable. Also, we remark that it is possible to prove that the set of equivalence classes of knots in $\mathbb{S}^3$ is countable, using result of 
Joyce \cite{J} and Matveev \cite{Matveev} that the fundamental quandle of a knot is almost complete invariant.

In Section \ref{Man} we are considering links in other manifolds and prove Theorem \ref{mt} using results from the paper \cite{L, HL, LR, LR1}, where links are presented as the closure  of braids and  some  analogue of Markov Theorem is proven.

In Section \ref{SS} we give another proof of Corollary \ref{cor}, using a fact that category of tangles is finitely generated. 

At the end of article we formulate a couple of questions for further research.

\section{Preliminaries}\label{prelim1}
\subsection{Dynamical systems}
A detailed explanation of the facts of this section can be found in \cite{S3}, \cite{GGZP}, \cite{GrMePo2016}.

Let $f:M^n\to M^n$ be a diffeomorphism on a smooth closed (compact without the  boundary) $n$-manifold ($n\geqslant 1$) $M^n$ with a metric $d$.  

Two diffeomorphisms $f, f':M^n\to M^n$ are said to be {\em topologically conjugate} if there is a homeomorphism $h:M^n\to M^n$ such that $fh=hf'$.

A point $x\in M^n$ is {\em wandering} for $f$ if there exists an open neighbourhood $U_x$ of $x$ for which $f^n(U_x)\cap U_x=\emptyset$ for every $n\in\mathbb N$. Otherwise $x$ is said to be {\em non-wandering}. The set of the non-wandering points of a diffeomorphism  $f$ is the {\em non-wandering set} denoted by $\Omega_f$.

If the set $\Omega_f$ is finite then each point $p\in\Omega_f$ is periodic. Denote by $m_p\in\mathbb N$ its period. Each periodic point $p$ has the corresponding {\em stable} and {\em unstable} manifolds defined by

$W^s_p=\{x\in M^n:\lim\limits_{k\to+\infty}d(f^{km_p}(x),p)=0\}$, 

$W^u_p=\{x\in M^n:\lim\limits_{k\to+\infty}d(f^{-km_p}(x),p)=0\}$.

Both the stable and the unstable manifolds are the {\em invariant manifolds}.

A periodic point $p\in\Omega_f$ is said to be {\em hyperbolic} if the absolute value of each eigenvalue of the Jacobian matrix  $\left(\frac{\partial f^{m_p}}{\partial x}\right)\vert_{p}$ is not equal to 1. If the absolute value of each eigenvalue is less (greater) than 1 then $p$ is the {\em sink} {\em (the source)}. Sinks and sources are called {\em nodes}. If a hyperbolic point is not a node then it is called a {\em saddle}.

It follows from the hyperbolic structure of a periodic point $p$ that its stable $W^s_p$ and unstable $W^u_p$ manifolds are the images of the spaces $\mathbb R^{q_p}$ and $\mathbb R^{n-q_p}$ by injective immersions where $q_p$ is the number of the eigenvalues of Jacobian matrix whose absolute value is greater than 1. A path connected component of the set $W^u_p\setminus p$ ($W^s_p\setminus p$) is called an {\em unstable (stable) separatrix} of the point $p$.

A closed $f$-invariant set $A\subset  M^n$ is called an {\em attractor} of the dynamical system $f$ if $A$ has a compact neighbourhood $U_A$ such that $f(U_A)\subset int\,U_A$ and $A=\bigcap\limits_{k\geqslant 0}f^k(U_A)$. In this case $U_A$ is called {\em trapping}  neighbourhood. 

A diffeomorphism $f:M^n\to M^n$ is said to be a {\em Morse-Smale diffeomorphism} if
\begin{enumerate}

\item the non-wandering set $\Omega_f$ consists of finite number of hyperbolic orbits;

\item for any two non-wandering points $p$, $r$ the manifolds $W^s_p$, $W^u_r$ intersect transversally.
\end{enumerate}



\subsection{Braids, links, and  tangles}\label{prelim2}
For the proof of Theorem \ref{mt} we need some facts on braids and links in 3-manifolds. In this subsection we recall them  (see, for example, \cite{Bir, CF, KT, K}).

\subsubsection{Braids} The Artin braid group $B_n$ is the group generated by $n-1$ generators $\sigma_1, \sigma_2, \ldots, \sigma_{n-1}$ which satisfy  the  relations
$$
\sigma_i \sigma_{i+1} \sigma_i =  \sigma_{i+1} \sigma_i \sigma_{i+1}~~\mbox{for}~i = 1, 2, \ldots, n-2,
$$
$$
\sigma_i \sigma_j = \sigma_j \sigma_i~~\mbox{for all}~i, j = 1, 2, \ldots, n-1~~\mbox{with}~|i - j| \geq 2.
$$

There is a unique group homomorphism $\pi \colon  B_n \to \Sigma_n$ of $B_n$ to the symmetric group $\Sigma_n$ such that $\pi(\sigma_i) = (i,i+1)$ for all 
$i = 1, 2, \ldots, n-1$.
On the other side, the formula $\iota(\sigma_i) = \sigma_{i}$ with
$i = 1, 2, \ldots, n-1$ defines an injective  group homomorphism
$$
\iota \colon B_n \to B_{n+1}.
$$

For the elements of $B_n$ it is possible to give a geometric interpretation. A geometric braid on $n \geq 1$ strings is a set $b \subset  {R}^2 \times I$, $I = [0, 1]$
formed by $n$  disjoint topological intervals called the strings of $b$ such that the projection $\mathbb{R}^2 \times I \to I$  maps each string homeomorphically onto $I$ and
$$
b \cap (\mathbb{R}^2 \times \{ 0 \}) = \{ (1, 0, 0), (2, 0, 0), \ldots, (n, 0, 0)\},
$$
$$
b \cap (\mathbb{R}^2 \times \{ 1 \}) = \{ (1, 0, 1), (2, 0, 1), \ldots, (n, 0, 1)\}.
$$

For any $\beta \in B_n$ the closed braid $\hat{\beta} \in \mathbb{S}^3$ is an oriented geometric  link. By the Alexander theorem any oriented link in $\mathbb{S}^3$
 is isotopic to a closed braid.
 
The infinite braid group $B_{\infty} = \cup_{n=1}^{\infty} B_n$ is a foundation of the knot theory in $\mathbb{S}^3$. The same role in studying  links in 3-manifolds  (see \cite{L}) plays the extended braid groups $B_{m,n}$, whose elements are called mixed braids and their  first $m$ strands forming the identity subbraid. The group $B_{m,n}$ is a subgroup of $B_{m+n}$ and has the presentation with the $m+n-1$ generators 
$$
a_1, a_2, \ldots, a_{m}, \sigma_1, \sigma_2, \ldots, \sigma_{n-1}
$$
 which satisfy  the  relations
$$
\sigma_i \sigma_{i+1} \sigma_i =  \sigma_{i+1} \sigma_i \sigma_{i+1}~~\mbox{for}~i = 1, 2, \ldots, n-2,
$$
$$
\sigma_i \sigma_j = \sigma_j \sigma_i~~\mbox{for all}~i, j = 1, 2, \ldots, n-1~~\mbox{with}~|i - j| \geq 2,
$$ 
$$
a_i \sigma_1 a_i  \sigma_1 =  \sigma_{1} a_i \sigma_{1} a_i~~\mbox{for}~1 \leq i \leq m,
$$
$$
a_i \sigma_k = \sigma_k a_i~~\mbox{for all}~k \geq 2,~~1 \leq i \leq m,
$$ 
$$
a_i (\sigma_1 a_r  \sigma_1^{-1}) =  (\sigma_1 a_r  \sigma_1^{-1}) a_i~~\mbox{for}~1 \leq i \leq m.
$$
For the geometric interpretation of $a_i$ and $\sigma_j$ see \cite[Figure 7]{L}. The structure of $B_{m,n}$ is studied in \cite{B1}.

The groups $B_{m,n}$ are the appropriate braid structures for studying knots and links in the complement of the $m$-unlink or a connected sum of $m$ lens spaces of
type $L(p, 1)$ or a handlebody of genus $m$. For these manifolds some  analogue of Markov Theorem holds (see \cite{L, HL, LR, LR1}). 

\subsubsection{Knots and links}

Recall, that {\it a knot} in a  3-manifold  $M^3$ is a homeomorphism $\gamma\colon \mathbb{S}^1\to  M^3$ or the image $K=\gamma(\mathbb{S}^1)$ of the  homeomorphism, where
$$
\mathbb{S}^1 = \{ z \in \mathbb{C} ~|~|z| = 1 \}
$$
is the standard unit circle.
 A knot $K$ in $M^3$ is said to be  {\it  essential} if  there is no a 3-ball $D^3 \subseteq M^3$ such that $K \subset D^3$.

We are considering only tame knots.
 A knot $K$ is said to be  {\it tame} if it is  equivalent to a polygonal knot; otherwise it is wild.  A {\it polygonal knot} is one which is the union of a
finite number of closed straight-line segments called edges, whose endpoints are the vertices of the knot.
Two knots $K,\,K'$ are called {\it equivalent} if there exists an orientation-preserving homeomorphism $h\colon   M^3\to M^3$ such that $h(K)=K'$. 

A {\it link} $L$ in a manifold $M^3$ is a disjoint union of finite number of knots. A link is called {\it essential (tame)} if it consists of essential (tame) knots.

\subsubsection{Tangles} Tangles are generalization of braids and links. For any integer $n > 0$ we put $[n] = \{ 1, 2, \ldots, n \}$ and assume that $[0]$ is the empty set.
If $k$ and $l$ are nonnegative integers. A tangle $L$ of type $(k,l)$ is the union of a finite number of pairwise disjoint simple oriented polygonal arcs in $X = \mathbb{R}^2 \times I$ such that the boundary $\partial L$ of $L$ satisfies the condition
$$
\partial L = L \cap (\mathbb{R}^2 \times \{ 0, 1 \}) = ([k] \times \{ 0 \} \times \{ 0 \}) \cup ([l] \times \{ 0 \} \times \{ 1 \}). 
$$
Observe that a link in $\mathbb{R}^2 \times I$ is a tangle of type $(0,0)$. As for links on the set of tangles it is possible to define an equivalent relation. More accurately, two tangles $L$ and $L'$ are equivalent an isotopy $h$ of $X$ such that $h(1, L) = L'$, where an isotopy of $X$ is a piecewise-linear map $h \colon I \times X \to X$ such that for all $t \in I$, the mapping $h(t, -)$ is a homeomorphism of $X$ restricting to the identity map on the boundary $\partial X = \mathbb{R}^2 \times \{ 0, 1 \}$ and such that $h(0, -)$ is the identity of $X$.

\bigskip


\section{Dynamic of class $G$ diffeomorphisms} \label{Dyn}
\subsection{Diffeomorphism scheme $f\in G$}
In this section we introduce the notion of a scheme of a diffeomorphism $f\in G$ whose equivalence class is a complete invariant of the topological conjugacy of the diffeomorphism \cite{BGP}. 

Let $V_{\alpha}=W^u_{\alpha}\setminus\alpha$. Due to the hyperbolicity of the source $\alpha$, there is a diffeomorphism $\psi_{\alpha}\colon V_{\alpha}\to\mathbb R^3\setminus O$, which conjugates the diffeomorphisms $f$ and $h^{-1}$. Let $p_{\alpha}=p\psi_{\alpha}: V_{\alpha}\to\mathbb S^{2}\times\mathbb S^1$ and $\mathcal T_f=p_{\alpha}(W^s_{\Omega_1}\cap V_\alpha)$. Pair $$S_f=(\mathbb S^{2}\times\mathbb S^1,\mathcal T_f)$$ is called a {\it scheme} of the diffeomorphism $f\in G$. According to \cite{GrMePo2016}, the set $\mathcal T_f$ consists of 
homotopically non-trivial two-dimensional tori, smoothly embedded into $\mathbb S^2\times\mathbb S^1$ (see Fig. \ref{Pi7}).
\begin{figure}[h!]\center{\includegraphics
		[width=0.45\linewidth]{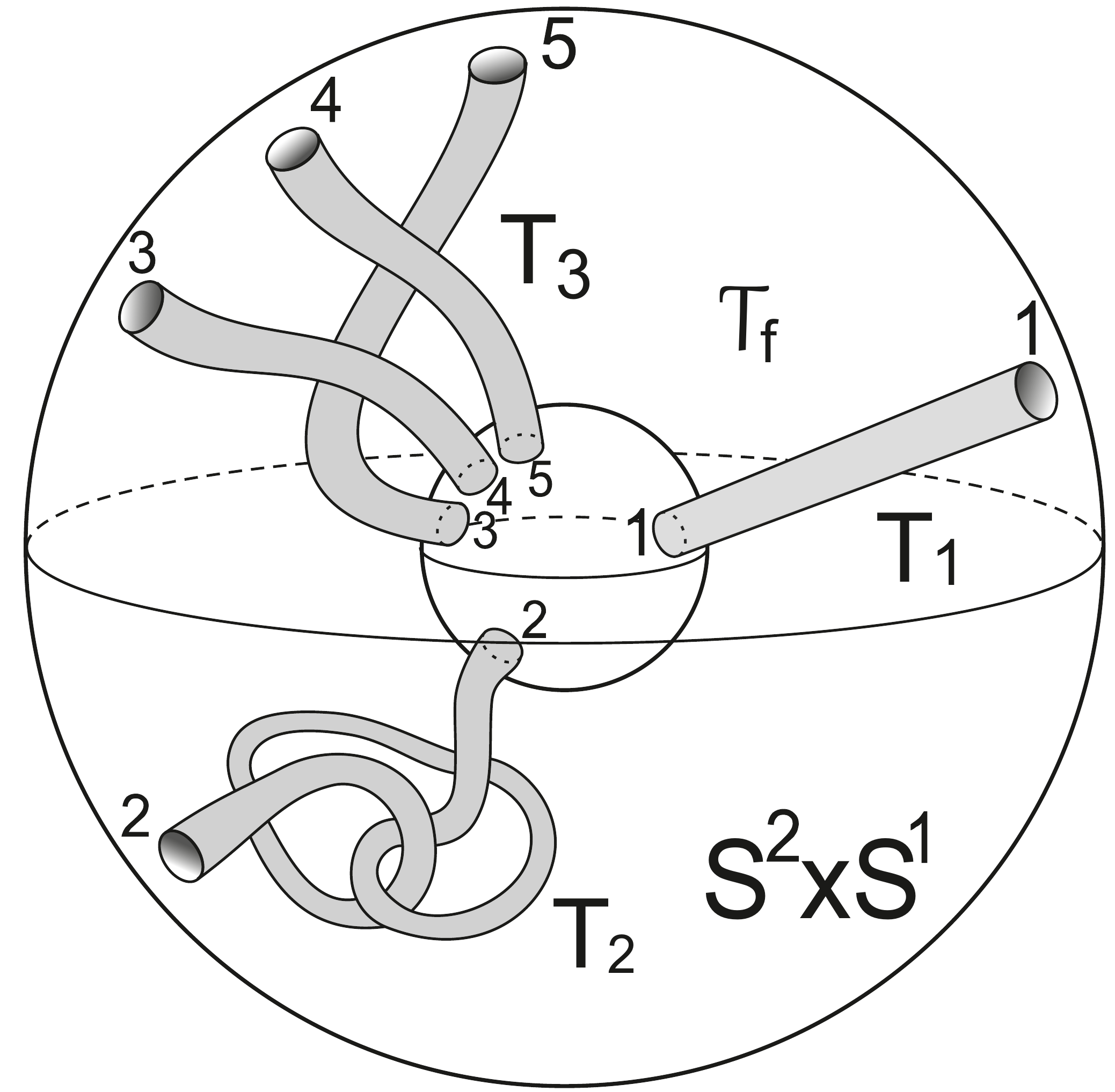}}\caption{Scheme $S_f$  of the diffeomorphism $f\in G$ from figure  \ref{pi1}}\label{Pi7}
\end{figure}

Two schemes $S_f,\,S_{f'}$ are called {\it equivalent} if there is a homeomorphism $\psi:\mathbb S^2\times\mathbb S^1\to \mathbb S^2\times\mathbb S^1$ such that $\psi(\mathcal T_f)=\mathcal T_{f'}$.
\begin{proposition}[\cite{BGP}, Theorem 1]\label{duke} Diffeomorphisms $f,\,f'\in G$ are topologically conjugate if and only if their schemes $S_f,\,S_{f'}$ are equivalent.
\end{proposition}

Thus, to prove theorem \ref{coni}, it is sufficient to show that for diffeomorphisms of the set $G$, the equivalence of links is tantamount to the equivalence of schemes.

\subsection{The equivalence of the links is tantamount to the equivalence of the schemes}
In this section we will prove a result (lemma \ref{bas}), directly from which the theorem \ref{coni} will follow.

\begin{lemma}\label{bas} The links of $\mathcal K_f,\,\mathcal K_{f'}$ of diffeomorphisms $f,\,f'\in G$ are equivalent if and only if their schemes $S_f,\,S_{f'}$ are equivalent.
\end{lemma}
\begin{proof} For $i\in\{1,\dots,k\}$ let $K_i=p(W^u_{\sigma_i}\setminus\sigma_i),\,T_i=p_\alpha(W^s_{\sigma_i}\setminus\sigma_i)$. Then $$\mathcal K_f=\bigsqcup\limits_{i=1}^kK_i,\,\,\mathcal T_f=\bigsqcup\limits_{i=1}^kT_i.$$
Denote by $U(K_i)$ a tubular neighbourhood of the knot $K_i$, homeomorphic to the solid torus. Then $U(T_i)=q_{_f}(p_f^{-1}(U(K_i)))\cup T_i$ be a one-sided tubular neighbourhood of the torus $T_i$. Let 
	$$U(\mathcal K_f)=\bigsqcup\limits_{i=1}^kU(K_i),\,\,U(\mathcal T_f)=\bigsqcup\limits_{i=1}^kU(T_i),$$  
	$$V=\mathbb S^2\times\mathbb S^1\setminus{\rm int}\,U(\mathcal K_f),\,\,W=\mathbb S^2\times\mathbb S^1\setminus{\rm int}\,U(\mathcal T_f).$$
By the construction the map $\xi=p_\alpha p_{\omega}^{-1}|_{V}:V\to W$ is a homeomorphism. We show that the homeomorphism $\xi$ continues to the homeomorphism $\xi:\mathbb S^2\times\mathbb S^1\to\mathbb S^2\times\mathbb S^1$. 
	
For any closed curve $c\subset\mathbb S^2\times\mathbb S^1$ we will denote by $\langle c\rangle$ the winding number of the loop $c$ around the generator of the  fundamental group $\mathbb S^2\times\mathbb S^1$. Then the torus $\partial V_i=\partial V\cap U(K_i)$ has generators $\lambda_i,\mu_i$ such that $$\langle\lambda_i\rangle=m_i,\,\langle\mu_i\rangle=0.$$ From the construction of the homeomorphism $\xi$, it follows that $\xi(\lambda_i),\xi(\mu_i)$ are the generators of the torus $\partial W_i=\partial W\cap U(T_i)$, while $$\langle\xi(\lambda_i)\rangle=m_i,\,\langle\xi(\mu_i)\rangle=0.$$  
	From the conditions imposed on the class $G$, it follows that each torus $\partial W_i$ bounds the solid torus $W_i$ in $W$. Then  
	$\xi(\mu_i)$ is the meridian of the solid torus $W_i$, and the homeomorphism $\xi|_{\partial V_i}:\partial V_i\to\partial W_i$ translates the meridian of the solid torus $U(K_i)$ to the meridian of the solid torus $W_i$, by virtue of which it continues to the solid tori \cite{Ro}.
	
	Thus, there is a homeomorphism $\xi:\mathbb S^2\times\mathbb S^1\to\mathbb S^2\times\mathbb S^1$ such that $\xi(U(K_i))=W_i,\,i\in\{1,\dots,k\}$. 
	
	We introduce similar notation with a prime for the diffeomorphism $f'\in G$. Without loss of the generality, we assume that $U(\mathcal K_{f'})=\varphi(U(\mathcal K_f))$. Let $\psi=\xi'\varphi\xi^{-1}:\mathbb S^2\times\mathbb S^1\to\mathbb S^2\times\mathbb S^1$.	Then $\psi(W_i)=W'_i$ and, without loss of the generality, we can assume that $\psi(T_i)=T'_i$. 
\end{proof}

\subsection{Realization of a link by a diffeomorphism}
In this section, we prove Theorem~\ref{rea}, namely, we realize any  link $\mathcal K\subset\mathbb S^2\times\mathbb S^1$ consisting of essential knots 
$$\mathcal K=K_1\sqcup\dots\sqcup K_k
$$ 
by a diffeomorphism $f\in G$. 

Let $\langle K_i\rangle=m_i$. Then the set $\bar K_i=p^{-1}(K_i)$ consists of $h$-invariant union of arcs 
$$\bar K_i=\bar K^0_i\sqcup h(\bar K^0_i) \sqcup \dots\sqcup h^{m_i-1}(\bar K^0_i).
$$
Let $U(K_i)\subset\mathbb S^2\times\mathbb S^1$ be a tubular neighbourhood of the knot $K_i$. Then $U(\bar K_i)=p^{-1}(U(K_i))$ is an $h$-invariant neighbourhood of the arcs $\bar K_i$ consisting of $m_i$ connected components $U(\bar K^0_i)\sqcup\dots\sqcup h^{m_i-1}(U(\bar K^0_i))$, each of which is diffeomorphic to $\mathbb{D}^{2}\times\mathbb R^1$.

Let {$C=\{(x_1,x_2,x_3)\in\mathbb R^3~:~x_2^2+x_{3}^2\leqslant 4\}$} and let a flow $g^t:C\to C$ is defined as $$g^t(x_1,x_2,x_3)=(x_1+t,x_2,x_3).$$ Then there is a diffeomorphism ${\zeta}_i:{U(\bar K^0_i)}\to C$, which conjugates the diffeomorphisms $h^{m_i}\vert_{{U(\bar K^0_i)}}$ and $g=g^1|_C$. Let's define a flow $\phi^t$ on $C$ by the following formulas:  
$$\begin{cases}
	\dot{x}_1=\begin{cases}1-\frac{1}{9}(x_1^2+x_2^2+x_3^2-4)^2, \quad x_1^2+x_2^2+x_3^2 \leqslant 4 \cr
		1, \quad x_1^2+x_2^2+x_3^2 > 4
	\end{cases}\cr
	\dot{x}_2=\begin{cases}
		\frac{x_2}{2}\big(\sin\big(\frac{\pi}{2}\big(x_1^2+x_2^2+x_3^2-3\big)\big)-1\big), \quad 2<x_1^2+x_2^2+x_3^2\leqslant 4\cr
		-x_2,\quad \quad x_1^2+x_2^2+x_3^2\leqslant 2\cr
		0, \quad x_1^2+x_2^2+x_3^2 > 4
	\end{cases}\cr
	\dot{x}_3=\begin{cases}
		\frac{x_3}{2}\big(\sin\big(\frac{\pi}{2}\big(x_1^2+x_2^2+x_3^2-3\big)\big)-1\big), \quad 2<x_1^2+x_2^2+x_3^2\leqslant 4\cr
		-x_3,\quad \quad x_1^2+x_2^2+x_3^2\leqslant 2\cr
		0, \quad x_1^2+x_2^2+x_3^2 > 4.
	\end{cases}
\end{cases}$$ 
\begin{figure}[h!]\centerline{\includegraphics	[width=8 true cm, height=5.5 true cm]{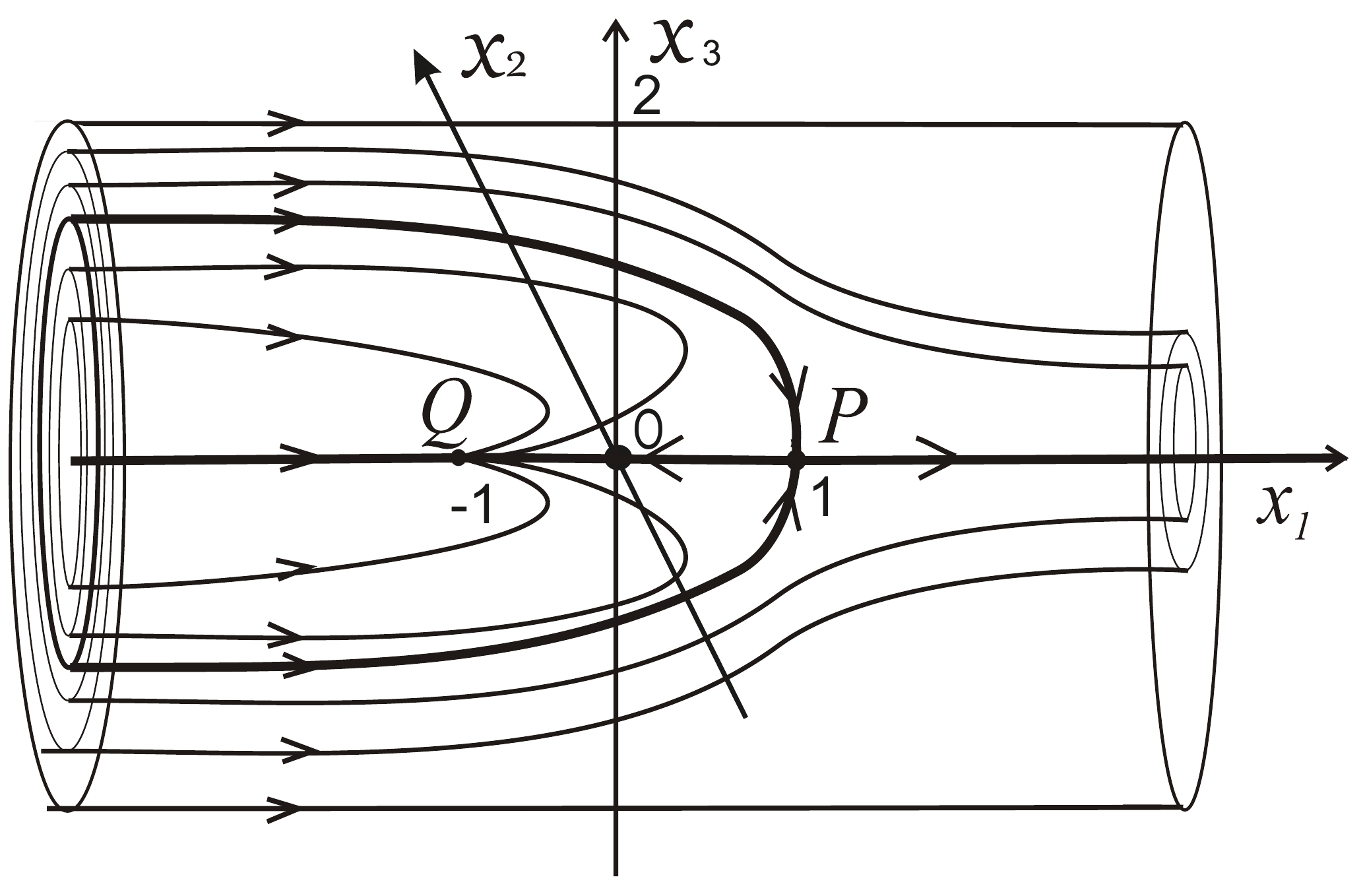}}\caption{\small Trajectories of the flow $\phi^t$}\label{cherry}	
\end{figure}
{By the construction, the diffeomorphism $\phi=\phi^1$ has two fixed points: the saddle $P(1,0,0)$ and the sink $Q(-1,0,0)$ (see Fig. \ref{cherry}), both hyperbolic. One unstable separatrix of the saddle $P$ coincides with the open interval $\left\{(x_1,x_2,x_3)\in\mathbb R^3:\,|x_1|<1,\,x_2=x_3=0\right\}$, belonging to the basin of the sink $Q$, and the other is the ray $\left\{(x_1,x_2,x_3)\in\mathbb R^3:\,x_1>1,\,x_2=x_3=0\right\}$. 
	Note that $\phi$ coincides with the diffeomorphism $g=g^1$ outside the ball $\{(x_1,x_2,x_3)\in C:x_1^2+x_2^2+x_3^2\leqslant 4\}$}. 

Define a diffeomorphism $\bar f_{\mathcal K}:\mathbb R^3\to\mathbb R^3$ in such a way that $\bar{f}_{\mathcal K}$ coincides with $h$ outside $U(\bar K_1)\cup\dots\cup U(\bar K_k)$ and coincides with a diffeomorphism $h_i:U(\bar K_i)\to U(\bar K_i)$, given by the formula $$h_i(h^j(x))=\begin{cases}h(h^j(x)),\,j\in\{0,\dots,m_i-2\},\,x\in U(\bar K^0_i);\cr {\zeta}_i^{-1}(\phi({\zeta}_i(x))),\,j=m_i-1,\,x\in U(\bar K^0_i).\end{cases}$$ 
Then $\bar f_{\mathcal K}$ has in $U(\bar K_i)$ two periodic orbits of the period $m_i$: of the sink $\omega_i={\zeta}_i^{-1}(Q)$ and of  the saddle $\sigma_i={\zeta}_i^{-1}(P)$, both hyperbolic (see Fig. \ref{af2}).
\begin{figure}[h!]\centerline{\includegraphics		[width=9 true cm, height=9 true cm]{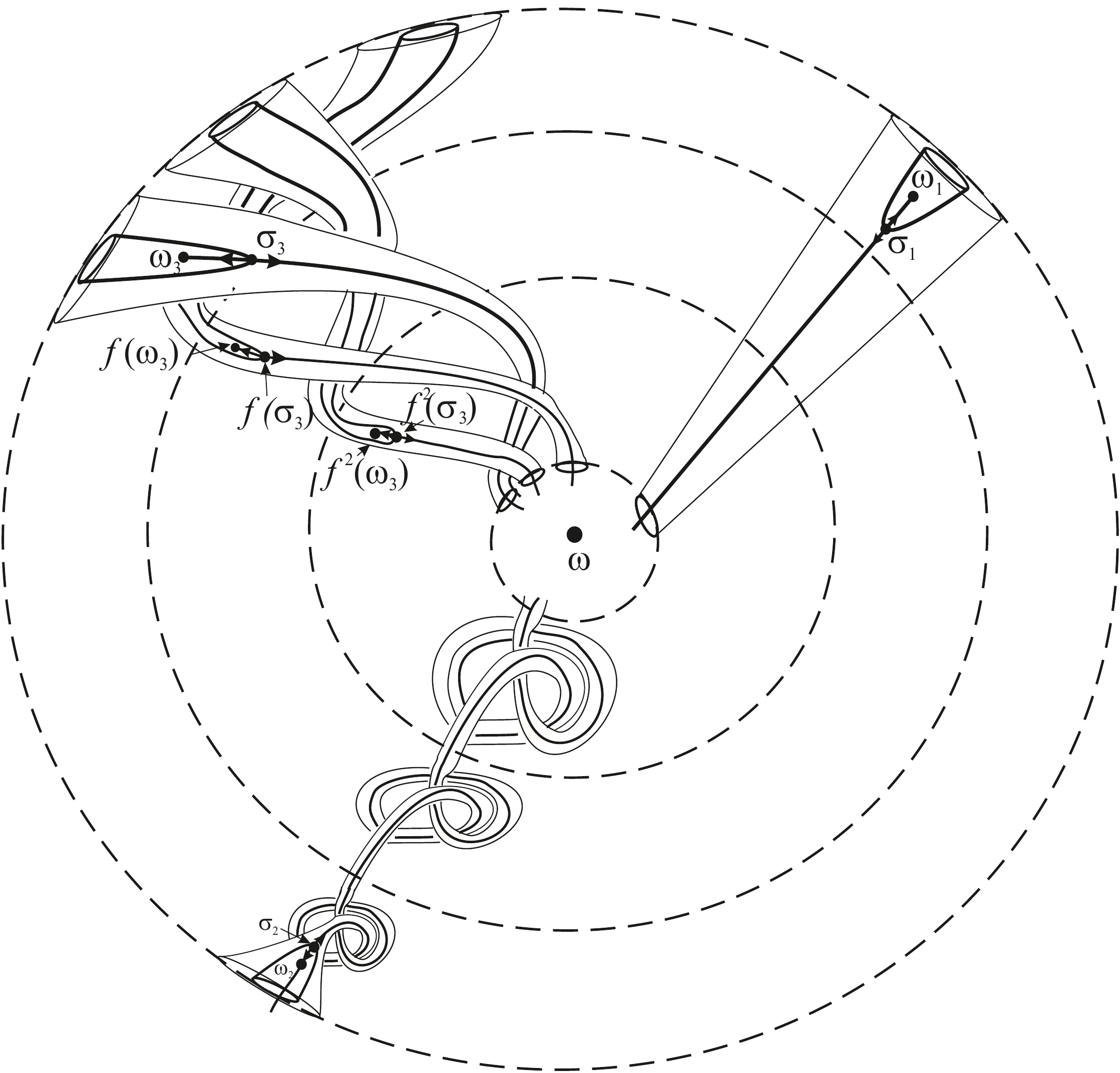}}\caption{\small Phase portrait of the diffeomorphism $\bar f_{\mathcal K}$}\label{af2}\end{figure}
{Now we project the dynamics onto the $3$-sphere. Denote by $N(0,0, 0, 1)$ the north pole of the sphere $\mathbb S^3=\{{\bf x}=(x_1,x_2,x_3,x_{4}):\,||{\bf x}||=1\}$. For each point ${\bf x}\in(\mathbb{S}^3\setminus\{N\})$ there is a single line passing through $N$ and ${\bf x}$ in $\mathbb R^{4}$ and this line intersects $\mathbb R^3$ at a single point $\vartheta_+({\bf x})$. The stereographic projection of the point ${\bf x}$ is the point $\vartheta_+({\bf x})$. It is easy to check that $$\vartheta_+(x_1,x_2,x_3,x_{4})=\left(\frac{x_1}{1-x_{4}},
	\frac{x_{2}}{1-x_{4}},\frac{x_3}{1-x_{4}}\right).$$ Thus, the stereographic projection $\vartheta_+:\mathbb S^3\setminus\{N\}\to\mathbb R^3$ is a diffeomorphism.}
{By the construction, the diffeomorphism $\bar{f}_{\mathcal K}$ coincides with $h$ in some neighbourhood of the initial point $O$ and the infinite point, therefore, it induces on $\mathbb{S}^3$ the Morse-Smale diffeomorphism $$f_{{\mathcal K}}({\bf x})=\begin{cases}\vartheta_+^{-1}(\bar f_{\mathcal K}(\vartheta_+({\bf x}))),~{\bf x}\neq N;\cr N,~{\bf x}=N\end{cases}.$$} It follows directly from the construction that the diffeomorphism $f_{{\mathcal K}}$ belongs to the set $G$, let's call such diffeomorphisms {\it model}. The following statement follows directly from theorem \ref{coni}.
\begin{corollary} $ $
	\begin{itemize}
\item Any diffeomorphism $f\in G$ is topologically conjugate to some model diffeomorphism $f_{\mathcal K}$.
		\item Model diffeomorphisms $f_{\mathcal K},f_{\mathcal K'}$ are topologically conjugate if and only if their links $\mathcal K,\,\mathcal K'$ are equivalent.
	\end{itemize} \label{Pidy}
\end{corollary}

\bigskip


\section{Counting of links and spatial graphs  in $\mathbb{S}^3$} \label{Sph}

In this section we prove Theorem  \ref{mt} for the case $M^3 = \mathbb{S}^3$. For further references we formulate it as proposition.

\begin{proposition} \label{S3}
The  set of  equivalence classes of tame links in $S^3$ is countable.
\end{proposition}

\begin{proof} Let   $B_{\infty} = \cup_{n=1}^{\infty} B_n$ be the braid group on infinite number of strings.
Any element of $B_n$ can be present by a  word in the alphabet $\{ \sigma_1^{\pm 1}, \sigma_2^{\pm 1}, \ldots, \sigma_{n-1}^{\pm 1} \}$. The set 
of such words is countable. Hence, the set of elements of $B_n$ is countable and since the union of countable number of countable sets is countable, the set of elements of $B_{\infty}$ is countable. 

For any braid $\beta \in B_{\infty}$ we can construct a corresponding geometric braid $b$ and take its closure $\hat{b}$, we get a link. By Alexander's theorem any tame  link in $\mathbb{S}^3$ can be construct by this procedure. Furthermore, by Markov's theorem, two links in $\mathbb{S}^3$ are equivalent if and only if there is a finite sequence of the so called Markov's moves, which sends a braid which represents the first link to the braid which represents the second link. In particular, the set of equivalence classes of tame links in  $\mathbb{S}^3$ is not bigger than the set of braids  in $B_{\infty}$, hence, it is countable. 
\end{proof}

\begin{remark} To prove that the number of equivalence classes of knots in $\mathbb{S}^3$ is countable, we can use another approach.  Joyce \cite{J} and Matveev \cite{Matveev} defined an algebraic system with one binary algebraic operation, which now is called a quandle (Matveev called it by a right invertible self-distributive groupoid). For any knot diagram $D_K$ of a knot $K$ in  $\mathbb{S}^3$  Joyce  and Matveev define a fundamental quandle $Q(D_K)$ and proved  that $Q(D_K)$ is an invariant of $K$, i.e. does not depend  on the diagram $D_K$.

 Let us show that the set of fundamental quandles is countable. Indeed, suppose that a diagram $D_K$ has $n$ crossings. Since, $D_K$ is a preimage of a circle,  it has  $n$ connected components. The quandle $Q(D_K)$ is generated by elements $x_1$, $x_2$, $\ldots$, $x_n$, one generator for one connected component, and is defined by $n$ relations (one for each crossing) of the form $x_k = x_i * x_j$,
if in this crossing met three connected components with labels $x_i$, $x_j$ and $x_k$. From this follows  that for fixed $n$ there are only finite number of fundamental quandles. Hence, there are countable number 
of fundamental quandles of tame knots in $\mathbb{S}^3$. Since, fundamental quandle is an almost complete  knot invariant (it means that if for knots $K$ and $J$ their fundamental quandles $Q(D_K)$ and $Q(D_J)$ are isomorphic, then $K$ is equivalent to $J$ or $K$ is equivalent to $-\bar{J}$, the mirror image of $J$ with the opposite orientation). From these observations follows the set of equivalence classes of  tame knots in  $\mathbb{S}^3$ is countable.
\end{remark} 

\medskip
 
A spatial graph $\Gamma$ (see \cite{K}) is a geometric realization of a combinatorial graph  in $\mathbb{S}^3$.
Two spatial graphs $\Gamma$ and $\Gamma'$ are  equivalent if there is an orientation-preserving homeomorphism $h \colon \mathbb{S}^3 \longrightarrow \mathbb{S}^3$ sending $\Gamma$ onto $\Gamma'$. A fundamental topological problem on spatial graphs is:
by an effective method, decide whether or not two given spatial graphs of a combinatorial graph are equivalent.

The theory of spatial graphs  is a generalization of  link theory  in 3-dimensional space.
L.~Kauffman \cite{Kauf} and independently D.~Yetter \cite{Y} defined an  equivalence relation on the set of spatial graphs and proved some analogous of Reidemeister theorem. 
 From Proposition 
\ref{S3} we get.

\begin{corollary}
The set of equivalence classes of  finite spatial graphs in the 3-sphere  $\mathbb{S}^3$ is countable.
\end{corollary}

\begin{proof}
 As was proved in \cite{BK} any  finite  spatial graph is a connected sum of a finite planar graph and  a braid. Since the set of finite planar graphs and  the set of braids are countable, we get the need assertion. 
\end{proof}

\bigskip


\section{Counting of links in 3-manifolds} \label{Man}

In this section we  consider links in 3-manifolds. 

{\it Proof of Theorem \ref{mt}.}
Let $H_g$ denote a handlebody of genus $g$. It usually is defined as 
$$
(\mbox{a closed disc} \setminus \{g~ \mbox{open discs} \}) \times I,
$$
 where $I$ is the unit interval. Equivalently, $H_g$ can be defined as
$$
( \mathbb{S}^3 \setminus ~\mbox{an open tubular neighbourhood of}~I_g),
$$
 where $I_g$ denotes the disjoint union of $g$ copies the interval $I$, all meeting at the point at infinity. Thus $H_g$ may be represented in  
 $\mathbb{S}^3$ by the trivial $g$-strand braid $I_g$ (see \cite{HL}). A   link $L$ in $H_g$ will be represented by the mixed link $I_g \cup L$ in $\mathbb{S}^3$. Two oriented links $L_1$, $L_2$ in $H_g$  are equivalent  in $H_g$ if and only if for the mixed links $I_g \cup L_1$ and $I_g \cup L_2$ there exists a homeomorphism $h \colon \mathbb{S}^3 \to \mathbb{S}^3$ which save the orientation, $h(I_g \cup L_1) = I_g \cup L_2$ and $h$ keeps $I_g$ pointwise fixed. 

As was proved in \cite{HL} any link in $H_g$ can be presented as the closure  of an algebraic mixed braid $I_g \cup B \in B_{g,n} \subseteq B_{g+n}$, denoted $I_g \cup\hat{B}$, is defined by joining each pair of the corresponding endpoints of the $B_n$-part by a vertical segment. Hence, any link in $H_g$ can be present by some braid which lies in $B_{g,n}$ for some $g$. Hence, by the result of the previous section we proved Theorem \ref{mt} for the case $M^3 = H_g$.


 Let now $M^3$ be the  complement of a link in $\mathbb{S}^3$. By the Alexander Theorem, this link is isotopic to the closure $\hat{B}$ of a braid $B$.
  So, we can write $M^3 = \mathbb{S}^3 \setminus \hat{B}$, it means that $M^3$ can be represented in $\mathbb{S}^3$ by $\hat{B}$. Further, let $M^3$ be a closed, connected, 
oriented 3-manifold (c.c.o. 3-manifold for simplicity). By results of Lickorish \cite{Lic} and Wallace \cite{W} $M^3$ can be
obtained from $\mathbb{S}^3$ by surgery along a framed link with integral framings. Without
loss of generality the surgery link can be assumed to be the closure $\hat{B}$ of a surgery
braid $B$. The framing of $\hat{B}$ induces a framing on the surgery braid $B$.
So, we can write $M^3 = \chi(\mathbb{S}^3, \hat{B})$ and $M^3$ can be represented in $\mathbb{S}^3$ by $\hat{B}$. Moreover, as was proved 
 in \cite{Lic}, all components of the surgery link can be assumed unknotted and  they can be equivalent to the closure of a pure braid. Thus, for
c.c.o.  3-manifolds we may assume $B$ to be a pure braid.

Suppose  now that $L$ is an oriented link in $M^3 = \mathbb{S}^3 \setminus \hat{B}$ or $\chi(\mathbb{S}^3, \hat{B})$. Fixing $\hat{B}$ pointwise we
may represent $L$ in $\mathbb{S}^3$ unambiguously by the mixed link $\hat{B} \cup L$, which consists of the fixed part $\hat{B}$ and the moving part $L$ that links 
with $\hat{B}$. A mixed link diagram is a diagram $\hat{B} \cup \tilde{L}$ of $\hat{B} \cup L$ on the plane of $\hat{B}$.
This plane is equipped with the top-to-bottom direction of $B$. By the Alexander Theorem and as explained in \cite[Theorem 5.3]{LR}, a diagram $\hat{B} \cup \tilde{L}$ of $\hat{B} \cup L$ may be turned into a mixed braid $\hat{B} \cup \beta$ with equivalent closure. This
is a braid in $\mathbb{S}^3$ with two different sets of strands.  One of the two sets comprises the fixed subbraid $B$, the other set of strands representing the link in the manifold $M^3$. So, $M^3$ may be represented in  $\mathbb{S}^3$ by the open braid $B$.

Further,  $L_1$ and $L_2$ are equivalent in $\mathbb{S}^3 \setminus \hat{B}$ 
if and only if the corresponding mixed links $\hat{B} \cup L_1$ and $\hat{B} \cup L_2$ are equivalent in $\mathbb{S}^3$ by a save orientation  homeomorphism of $\mathbb{S}^3$ which keeps $\hat{B}$ pointwise fixed.

Hence, if $M^3$ is a 3-manifold from Theorem \ref{mt}, then any link $L \subset M^3$ can be present by a braid from $B_{\infty}$, which depends on $M^3$. Since the set of such braids is countable and us was shown in \cite{L, HL, LR, LR1} the set of equivalence classes of links in $M^3$ is the quotient of braids by some generalization of Markov moves,   the set of equivalent classes of links is countable. The theorem is proved.

\bigskip


\section{Links in $\mathbb S^2\times\mathbb S^1$} \label{SS}

The manifold $\mathbb S^2\times\mathbb S^1$ plays important role in  dynamical systems.  Knot theory in this manifold  is the same as Knot theory  in the solid torus (handlebody of genus 1) or Knot theory in the complement of the unknot in $\mathbb S^3$. Hence, Corollary \ref{cor} follows from Theorem  \ref{mt}. On other side,  it is possible to give an independent proof. 

{\it Independent proof of Corollary \ref{cor}.} We shall use the fact that arbitrary link $L$ in $\mathbb{S}^2 \times \mathbb{S}^1$, after an isotopy, can be  draw as the closure of a tangle $P_L$ (see, for example, \cite[Figure 5]{DNPR}), which we can assume a tangle in  $\mathbb S^3$. Hence, if we show that the set of tangles, which represent tame links is countable, then the corollary will be proven. 

We will use some facts on the category of tangles, which can be found in \cite[Chapter 12]{Kas}. Tangles form a strict tensor category $\mathcal{T}$ as follows. The objects of 
$\mathcal{T}$ consist of finite sequence of $\pm$ signs, including the empty sequence $\emptyset$, and the morphisms of $\mathcal{T}$ are the isotopy classes of oriented tangles. The composition of tangles, $L \circ L'$ is obtained by placing $L'$ on top of $L$. Also, it is possible to introduce a tensor product on $\mathcal{T}$. If $L$ and $L'$ are isotopy classes of oriented tangles, $L \otimes L'$ is the isotopy class of the oriented tangle obtained by placing $L'$ to the right of $L$. The tangle category $\mathcal{T}$ equipped with the tensor product is a strict tensor category in, which, the unit $I$ is the empty set $\emptyset$. The strict tensor category $\mathcal{T}$ is generated by the six morphisms
$$
\cup, ~~\overset{\gets}{\cup}, ~~\cap, ~~\overset{\gets}{\cap},  ~~X_{+},  ~~ X_{-},
$$
and is defined by finite set of relations (see \cite[Theorem 12.2.2]{Kas}). This means that the set of tangles which closures correspond to oriented links in
 $\mathbb{S}^2 \times \mathbb{S}^1$  is countable.

To see that the set of equivalence classes of essential links in  $\mathbb{S}^2 \times \mathbb{S}^1$ is infinite we can present any link in this manifold as a link in $\mathbb{S}^3$ with added unknot $U$. It is easy to see that if the linking number $lk(U, K_i)$ is not equal to zero for some component $K_i$ of the link, then this link is essential. Since the number of links in  $\mathbb{S}^3$ is countable, the set of essential links in $\mathbb{S}^2 \times \mathbb{S}^1$ is infinite and countable. 

For studying links in $\mathbb{S}^2 \times \mathbb{S}^1$ we can present them as closures of braids in $B_{1,\infty}$ and use analogous of Markov theorem \cite{HL}. On this way it is possible to define some polynomial invariants and construct non-equivalent links. On the other side, links in $\mathbb{S}^2 \times \mathbb{S}^1$ can be presented as links in the solid torus $V = D^2 \times S^1$.
Although every braid $\beta$ on $n$ strings gives rise to a closed $n$-braid in the solid
torus and  two closed braids in $V$  are isotopic if they are isotopic as oriented links, in general, a link in $V$ is not isotopic 
to a closed braid in $V$. For instance, a link lying inside a small 3-ball in $V$ is never isotopic to a closed braid. But by Theorem 2.1 \cite{KT}, for any $n \geq 1$ and any $\beta$,
$\beta' \in B_n$, the closed braids $\hat{\beta}$, $\hat{\beta'}$ are isotopic in the solid torus if and only if $\beta$ and $\beta'$ are conjugate in $B_n$. Well known that the conjugacy problem is solvable  in the braid group.

\medskip

At the end we formulate the next questions.

\begin{question} $ $
 Is there a 3-manifold   (or a topological space) $M$ such that the set equivalence classes of links  in $M$ is uncountable? If these manifold (topological space)  exist what can we say on the fundamental group $\pi_1(M)$?
\end{question}

\begin{question} $ $
We know that a link $L$ in some 3-manifold $M^3$ can be present by mixed braids. Is it possible to find the fundamental  group $\pi_1 (M^3 \setminus L)$ of the complement $L$ in $M^3$ find, using the Artin representation of the braid groups $B_n$ by automorphisms of free groups (as in the case of links in 
$\mathbb{S}^3$)?
\end{question}

\section*{Acknowledgments}
The topological part of this work was supported by the Ministry of Science and Higher Education of Russia (agreement No. 075-02-2023-943). The dunamical part is supported by the RSF (project No. 21-11-00010). The authors thank Andrey Malyutin and Misha Neshchadim for useful discussions. Also, the authors thank the organizers of the Conference December Days in Tomsk (December 2022) for this beautiful conference, where we started  to work on this article.

\bigskip

Valeriy Bardakov$^{\dag, \ddag, ^*,**}$ (bardakov@math.nsc.ru),\\

Tatyana Kozlovskaya $^{\dag}$(t.kozlovskaya@math.tsu.ru),\\

Olga Pochinka $^{***}$ (olga-pochinka@yandex.ru)

~\\
$^{\dag}$ Tomsk State University, pr. Lenina 36, 634050 Tomsk, Russia,\\
$^{\ddag}$ Sobolev Institute of Mathematics, Acad. Koptyug avenue 4, 630090 Novosibirsk, Russia,\\
$^{**}$ Novosibirsk State Agrarian University, Dobrolyubova  160, 630039 Novosibirsk,  Russia.\\
$^{***}$ HSE,  B. Pecherskaya, 25/12,  603155, Nizhny Novgorod, Russia. 
\end{document}